\documentclass{amsart}

\usepackage{hyperref}
\hypersetup{%
  colorlinks=true,
  linkcolor=black,
  citecolor=black,
  urlcolor=blue,
  pdftitle={On the existence of global solutions for the $3$D chemorepulsion system},
  pdfauthor={},
  pdfkeywords={},
  bookmarksopen=true,
}
\usepackage{url}

\usepackage{enumerate}
\usepackage{amsthm,amsmath,amssymb}
\usepackage{mathtools}
\usepackage{accents}

\usepackage[polish,english]{babel}
\usepackage[utf8]{inputenc}
\usepackage[OT4]{fontenc}

\usepackage{fancyhdr}

\newtheorem{theorem}{Theorem}[section]
\newtheorem{lemma}[theorem]{Lemma}

\newtheorem{proposition}[theorem]{Proposition}

\theoremstyle{definition}

\theoremstyle{remark}
\newtheorem{remark}[theorem]{Remark}

\numberwithin{equation}{section}

\providecommand{\abs}[1]{\left\lvert{#1}\right\rvert}
\providecommand{\norm}[1]{\left\lVert{#1}\right\rVert}

\newcommand{\lr}[1]{\left({#1}\right)}
\newcommand{\lra}[1]{\left\lbrack{#1}\right\rbrack}

\def\div{\rm div}

\newcommand{\N}{{\ensuremath{\mathbb N}}}
\newcommand{\R}{{\ensuremath{\mathbb R}}}

\newcommand{\intom}{\int_\Omega}
\newcommand{\tmax}{T_{\max}}
\newcommand{\ure}{\mathrm{e}}
\newcommand{\leb}[1]{L^{#1}(\Omega)}
\newcommand{\sobo}[2]{W^{#1, #2}(\Omega)}
\newcommand{\Ombar}{\bar \Omega}

\newcommand{\ds}{\,\mathrm ds}

\newcommand{\hp}{\hphantom}
\newcommand{\pe}{\mathrel{\hp{=}}}

\newcommand{\defs}{\coloneqq}
\newcommand{\sfed}{\eqqcolon}

\def\<{{\langle}}
\def\>{{\rangle}}
\def\:{{\colon}}

\def\d{{\rm d}}
\def\e{{\rm e}}

\def\vro{{\varrho}}
\def\vreps{{\varepsilon}}

\def\b4{\frac{1}{4\pi}}
\def\2b4{\frac{1}{2\pi}}
\def\iR3{\int_{\R^3}}

\begin{document}

\title{On the existence of global solutions for the $3$d chemorepulsion system}

%    Information for first author
\author{Tomasz Cie\'slak} %\footnote{Supported by }
\address{Institute of Mathematics, Polish Academy of Sciences,
\'Sniadeckich 8, 00-656 Warsaw, Poland}
\email{cieslak@impan.pl}
%\thanks{}

%    Information for second author

\author{Mario Fuest}
\address{Leibniz  Universit\"at Hannover, Institut f\"ur Angewandte Mathematik, Welfengarten 1,
30167 Hannover, Germany}

\email{fuest@ifam.uni-hannover.de}

%    Information for third author

\author{Karol Hajduk} %\footnote{Supported by }
%    Address of record for the research reported here
\address{Institute of Mathematics, Polish Academy of Sciences,
\'Sniadeckich 8, 00-656 Warsaw, Poland}
%    Current address
%\curraddr{}
\email{khajduk@impan.pl}
%    \thanks will become a 1st page footnote.
%\thanks{KH was supported by }

%    Information for fourth author
\author{Miko{\L}aj Sier\.z{\k{e}}ga} %\footnote{Supported by}
\address{Faculty of Mathematics, Informatics and Mechanics, University of Warsaw,
Banacha 2, 02-097 Warsaw, Poland}
\email{m.sierzega@uw.edu.pl}
%\thanks{}

%    General info
\subjclass[2010]{Primary 35B45, 35K45; Secondary 92C17}

%\date{January 13, 2023.}

\keywords{boundedness of solutions, chemotaxis, chemorepulsion, Lyapunov-like functional, Bernis-type inequality}

\begin{abstract}
In this paper, we give sufficient conditions for global-in-time existence of classical solutions for the fully parabolic chemorepulsion system posed on a convex, bounded three-dimensional domain. Our main result establishes global-in-time existence of regular nonnegative solutions provided that $\nabla\sqrt{u} \in L^4(0, T; L^2(\Omega))$. Our method is related to the Bakry--\'Emery calculation and appears to be new in this context.
\end{abstract}

\maketitle

\section{Introduction}

\makeatletter
    \renewcommand{\theequation}{{\thesection}.\@arabic\c@equation}
    \renewcommand{\thetheorem}{{\thesection}.\@arabic\c@theorem}
    \renewcommand{\thedefinition}{{\thesection}.\@arabic\c@definition}
    \renewcommand{\thecorollary}{{\thesection}.\@arabic\c@corollary}
    \renewcommand{\theproposition}{{\thesection}.\@arabic\c@proposition}
\makeatother

In this paper, we study the problem of global existence of solutions of the fully parabolic chemorepulsion system. The two-dimensional case was solved in \cite{CLMR}. Unlike in the more-widely known chemoattraction case, $2$D chemorepulsion leads to the global-in-time existence of classical solutions regardless of the size of the initial data. The question of global existence in three and higher dimensions remains open. In the present paper, we look into the $3$D case and establish a conditional global regularity result for this model. First, we introduce the model.

Let $\Omega \subset \R^n$ be an open, bounded domain with a sufficiently smooth boundary. We consider the following fully parabolic chemorepulsion system
\begin{equation} \label{CRu}
\left\{\begin{split}
\partial_tu &= \nabla \cdot (\nabla u + u\nabla v) \\ %\label{CRv}
\partial_t v &= \Delta v - v + u
\end{split}\right. \qquad \mbox{in} \qquad (0, \infty) \times \Omega,
\end{equation}
with homogeneous Neumann boundary conditions (no-flux through the boundary)
\begin{align}\label{BC}
\left.\frac{\partial u}{\partial \nu}\right\rvert_{\partial\Omega} = 0 \qquad \& \qquad \left.\frac{\partial v}{\partial \nu}\right\rvert_{\partial\Omega} = 0,
\end{align}
where $\nu$ is the unit outward normal to the boundary, and with nonnegative initial conditions
\begin{align}\label{IC}
u(0, x) = u_0(x) \ge 0 \qquad \& \qquad v(0, x) = v_0(x) \ge 0.
\end{align}
The functions $u$ and $v$ describe densities of some living organisms and of a chemical substance which repels them, respectively. The `$+$' sign on the right-hand-side of the first equation in $(\ref{CRu})$ corresponds to the repulsion mechanism. The opposite phenomenon appears in
the widely studied chemoattraction case described by the Keller--Segel system.
For an overview of results for such systems, we refer to the surveys \cite{BBTW, LW}. The latter survey contains a chapter concerning the construction of solutions, including the irregular ones, to the fully parabolic chemorepulsion system.

Three and higher-dimensional cases of the chemorepulsion system are still far from being understood. While we know that global weak solutions exist (see \cite{CLMR}), it is unclear whether regular bounded solutions exist for all $t>0$. Some results concerning the perturbation of the parabolic-elliptic case are known, see \cite{CF}. Moreover, for the problem with nonlinear, sufficiently strong diffusion, see the global existence result in \cite{Freitag}. Similarly, it is known that nonlinear, sufficiently weak chemorepulsion leads to global-in-time solutions, see \cite{Tao}. However, the main basic problem lacks a definitive answer. In the present note we establish a conditional result.

We emphasise that our result is rather of methodological meaning. On the one hand, our method applied to the 3D chemorepulsion yields only a conditional result. Moreover, as communicated to us by M.~Winkler and one of the referees, this result can be improved, see Appendix~\ref{sec:alt_proof}. Indeed, the result in Appendix~\ref{sec:alt_proof} covers our Thereom~\ref{mainresult} and is applicable also in non-convex domains. Notice however, that our method of estimating the Fisher information along the solution of a system of partial differential equations seems promising in other type of problems. The inequality from Appendix A turns out to be very helpful in such an approach. Indeed, in \cite{BC} Fisher's information together with inequality (\ref{cieslakineq}) were succesfully applied to obtain global-in-time unique regular solutions to the 1D thermoelasticity problem. Very recently, another application yielding progress in 1D combustion theory was performed in \cite{LY}. Last, but not least, notice that our approach gives also an qualitative conditional result, namely it implies that concavity of $v$ is sufficient to obtain global solution to \eqref{CRu}, see Section~\ref{s:conclusion}.     

The problem \eqref{CRu}--\eqref{IC} captured attention of groups of  researchers, in particular due to its role in the attraction-repulsion competition, which appears to play a role in the modeling of, among other things, Alzheimer's disease, see for instance \cite{luca}. The biological meaning of the attraction-repulsion competition was widely investigated, see for instance the contributions in \cite{lin_mu, nagaial,TaoWang}.

Let us formulate our main result.

\begin{theorem}\label{mainresult}
Let $\Omega \subset \R^3$ be a convex, smooth bounded domain and let $u_0, v_0 \in W^{1, p}(\Omega)$ for some $p > 3$ with $0 \not\equiv u_0 \ge 0$ and $v_0 \ge 0$ in $\Omega$. Suppose that $\tmax \in (0, \infty]$ is the maximal existence time of the classical solution to the system \eqref{CRu}--\eqref{IC}, constructed in \cite[Theorem~2.1]{CLMR} (cf.\ Lemma~\ref{local_ex} below). If $\nabla\sqrt{u} \in L^4\lr{0, \tmax; L^2(\Omega)}$, then $
\tmax=\infty$.
\end{theorem}

The paper is organised as follows: In Section $\ref{s:prelim}$, we introduce some technical tools such as the Winkler version of the Bernis-type inequality, Bochner's formula, the behaviour of the normal derivative of the gradients of regular functions at the boundaries of convex domains and further preparatory inequalities. In Section~\ref{sec:local_ex}, we then recall well-known properties of solutions to \eqref{CRu}--\eqref{IC}.

Section~\ref{section_main} is devoted to our entropy estimate. It is a core and the main novelty of our approach. We estimate the time derivative of the entropy production term occurring in the Lyapunov functional. The latter entropy production term resembles the Fisher information along the heat flow and we utilize this similarity. Having this estimate established, in Section~\ref{fourth} we proceed with the proof of Theorem~\ref{mainresult}.

In Appendix~$\ref{s:bernis}$, we prove a new functional  inequality, which we arrived at as a byproduct of our investigations. It seems  interesting in its own right. Appendix~\ref{sec:alt_proof} is devoted to the proof of an observation due to M.\.Winkler and one of the anonymous referees, which improves the conditional result.

\section{Preliminaries} \label{s:prelim}

\makeatletter
    \renewcommand{\theequation}{{\thesection}.\@arabic\c@equation}
    \renewcommand{\thetheorem}{{\thesection}.\@arabic\c@theorem}
    \renewcommand{\thedefinition}{{\thesection}.\@arabic\c@definition}
    \renewcommand{\thecorollary}{{\thesection}.\@arabic\c@corollary}
    \renewcommand{\theproposition}{{\thesection}.\@arabic\c@proposition}
\makeatother

In this section, we collect some computations and known results which will be useful in the sequel.

We begin with the flat case of the well-known Bochner formula.
\begin{proposition}\label{Bochner}
Let $\Omega \subset \R^n$, $n \in \N$, be a smooth domain and let $u \in C^3(\bar \Omega)$. Then
\begin{equation}\label{bochnerformula}
\frac{1}{2}\Delta\abs{\nabla u}^2 = \nabla\lr{\Delta u}\cdot \nabla u +\abs{D^2u}^2
\quad \text{in}\quad \Ombar.
\end{equation}
\end{proposition}
\begin{proof}
This can be checked by a direct calculation.
\end{proof}

The following lemma informs us about the normal derivative of the square of the gradient of a function on the boundary of a convex set, provided the function's normal derivative vanishes.

\begin{lemma}\label{gradBCEvans}
Let $\Omega \subset \R^n$, $n \in \N$, be a convex bounded domain with smooth boundary. Suppose that a function $u \in C^2(\Ombar)$ satisfies
$\frac{\partial u}{\partial \nu}=0$ on $\partial\Omega$. Then
\[
\left.\frac{\partial\abs{\nabla u}^2}{\partial \nu}\right\rvert_{\partial\Omega} \le 0.
\]
\end{lemma}
\begin{proof}
  See \cite[p.~95]{E}.
\end{proof}

We will use a higher-dimensional version of the Bernis-type inequality given by Winkler \cite[Lemma~3.3]{Winkler} (with $h(\varphi) = \varphi$).
\begin{lemma}\label{winkler}
Let $\Omega \subset \R^n$, $n \in \N$, be a smooth, bounded domain.
For all positive $\varphi \in C^2\lr{\bar{\Omega}}$ with $\frac{\partial \varphi}{\partial \nu} = 0$ on $\partial\Omega$, we have the following inequality
\begin{align}\label{winklerineq}
\int_{\Omega}{\frac{|\nabla \varphi|^4}{\varphi^3}} \le \lr{2 + \sqrt{n}}^2\int_{\Omega}{\varphi \abs{D^2 \log \varphi}^2}.
\end{align}
\end{lemma}

Next, we prove two estimates holding in three-dimensional domains.
The first one relates the Hessian of a function $\varphi$ with $\nabla \Delta \varphi$ (in contrast to the full third-order derivative).
\begin{lemma}\label{lemmatwo}
Let $\Omega \subset \R^3$ be a smooth, bounded domain.
Then there is a positive constant $C$ such that for every $\varphi \in C^3(\bar{\Omega})$ with $\frac{\partial\varphi}{\partial \nu} = 0$ on $\partial\Omega$ we have
\begin{align}\nonumber
\norm{D^2\varphi}_{L^6(\Omega)} \le C\norm{\nabla\Delta\varphi}_{L^2(\Omega)}.
\end{align}
\end{lemma}
\begin{proof}
According to \cite[Theorem $19.1$]{Friedman}, there is $c_1 > 0$ such that
\begin{align}\nonumber
\norm{D^2\varphi}_{L^6} \le c_1\norm{\Delta\varphi}_{L^6} + c_1\norm{\varphi - \bar{\varphi}}_{L^6}
\end{align}
for every $\varphi \in C^2(\bar \Omega)$ with $\frac{\partial\varphi}{\partial \nu} = 0$ on $\partial\Omega$, where $\bar \varphi=\frac{1}{|\Omega|}\int_\Omega \varphi$. As $W^{1, 2}(\Omega) \hookrightarrow L^6(\Omega)$, we can further estimate
\begin{align}\nonumber
\norm{D^2\varphi}_{L^6} \le c_2\lr{\norm{\nabla\Delta\varphi}_{L^2}^2 + \norm{\Delta\varphi}_{L^2}^2}^{1/2} + c_2\lr{\norm{\nabla\varphi}_{L^2}^2 + \norm{\varphi - \bar{\varphi}}_{L^2}^2}^{1/2}
\end{align}
for every $\varphi \in C^3(\bar\Omega)$ with $\frac{\partial\varphi}{\partial \nu} = 0$ on $\partial\Omega$, for some $c_2 > 0$. The statement then follows by the Poincar\'e inequality. See, for example, \cite[Lemma~A.1]{Fuest} which states that there exists a constant $c > 0$ such that
\begin{align}\nonumber
\norm{\varphi - \bar{\varphi}}_{L^2}^2 \le c\norm{\nabla\varphi}_{L^2}^2, \quad \norm{\nabla\varphi}_{L^2}^2 \le c\norm{\Delta\varphi}_{L^2}^2 \quad \& \quad \norm{\Delta\varphi}_{L^2}^2 \le c \norm{\nabla\Delta\varphi}_{L^2}^2
\end{align}
for all $\varphi \in C^3(\bar \Omega)$.
\end{proof}

Finally, we combine several of the lemmata above to obtain an estimate required in the proof of our main result.
\begin{lemma}\label{lemmaone}
Let $\Omega \subset \R^3$ be a convex, smooth bounded domain and let $\vreps > 0$ and $M > 0$. Then there exists $C > 0$ such that for every $0 < \varphi \in C^2(\bar \Omega)$ with $\intom \varphi = M>0$ and $\psi \in C^3(\bar \Omega)$ that satisfy $\partial_\nu \varphi = \partial_\nu \psi = 0$ on $\partial \Omega$ we have
\begin{align} \nonumber
&\mathrel{\hphantom{=}} \int_{\Omega}|(\nabla \sqrt \varphi)^TD^2 \psi (\nabla \sqrt \varphi )| \\ &\le C\lr{\int_{\Omega}{\abs{\nabla \sqrt \varphi}^{2}}}^3 + C + \vreps\int_{\Omega}{\varphi \abs{D^2\log{\varphi}}^2} + \vreps\int_{\Omega}{\abs{\nabla\Delta \psi}^2}.
\label{notconvex}
\end{align}
\end{lemma}

\begin{proof}
By Hölder’s inequality, we have
\begin{align}\label{lemmaone:hoelder}
&\pe\int_{\Omega}|(\nabla\sqrt{\varphi})^T D^2 \psi(\nabla\sqrt{\varphi})| \nonumber \\
&\le \frac12 \int_{\Omega}{\abs{\nabla\sqrt{\varphi}}\abs{D^2 \psi}\frac{\abs{\nabla \varphi}}{\varphi^{3/4}}\varphi^{1/4}} \nonumber \\
&\le \frac12 \norm{\nabla \sqrt \varphi}_{L^2(\Omega)}\norm{D^2\psi}_{L^6(\Omega)}\norm{\frac{\nabla \varphi}{\varphi^\frac34}}_{L^4(\Omega)}\norm{\varphi^{1/4}}_{L^{12}(\Omega)}
\end{align}
for all $0 < \varphi \in C^2(\bar \Omega)$ and $\psi \in C^3(\bar \Omega)$.

Since $W^{1, 2}(\Omega) \hookrightarrow L^6(\Omega)$ and $\int_{\Omega}{\varphi} = M$ there is $c_1 > 0$ such that
\begin{align*}
\norm{\varphi^{1/4}}_{L^{12}(\Omega)} &= \norm{\varphi^{1/2}}^{1/2}_{L^6(\Omega)} \le c_1\norm{\varphi^{1/2}}^{1/2}_{W^{1,2}(\Omega)} \\
%= c\lr{\norm{\nabla \varphi^{1/2}}^{2}_{L^2} + \norm{\varphi^{1/2}}^{2}_{L^2}}^{1/4} \\ \nonumber
&\le c_1\norm{\nabla \varphi^{1/2}}^{1/2}_{L^2(\Omega)} + c_1\norm{\varphi^{1/2}}^{1/2}_{L^2(\Omega)} = c_1\lr{\norm{\nabla \sqrt \varphi}^{1/2}_{L^2(\Omega)} + M^{1/4}}
\end{align*}
for all $0 < \varphi \in C^2(\bar \Omega)$ with $\intom \varphi = M$.
In combination with \eqref{lemmaone:hoelder}, Lemma~\ref{lemmatwo}, Winkler's inequality \eqref{winklerineq}, the elementary estimate
\[
a(\sqrt{a} + \sqrt{b}) \le 2(a+b) \sqrt{a+b} \quad \text{for} \quad  a,b \ge 0
\]
and Young’s inequality, we see that by taking $a=\norm{\nabla \sqrt \varphi}_{L^2(\Omega)}$, $b=M^{1/2}$ and with some $c_2 > 0$ and $C >0$, we have
\begin{align*}
&\pe \int_{\Omega}|(\nabla \sqrt \varphi)^TD^2\psi\nabla(\sqrt \varphi)| \\
&\le c_2 \norm{\nabla \sqrt \varphi}_{L^2(\Omega)}
      \norm{\nabla \Delta \psi}_{L^2(\Omega)}
      \lr{\int_{\Omega}{\varphi\abs{D^2\log{\varphi}}^2}}^{1/4}
      \lr{\norm{\nabla \sqrt \varphi}^{1/2}_{L^2(\Omega)} + M^{1/4}} \\
&\le 2c_2 \lr{\norm{\nabla \sqrt \varphi}_{L^2} + M^{1/2}}^\frac32\lr{\int_{\Omega}{\varphi\abs{D^2\log{\varphi}}^2}}^{1/4} \norm{\nabla\Delta \psi}_{L^2} \\
&\le C\norm{\nabla \sqrt \varphi}^6_{L^2} + C + \vreps\int_{\Omega}{\varphi\abs{D^2\log{\varphi}}^2}
 + \vreps\norm{\nabla\Delta \psi}^2_{L^2}
\end{align*}
for all $0 < \varphi \in C^2(\bar \Omega)$ and $\psi \in C^3(\bar \Omega)$ with $\intom \varphi = M$ and $\partial_\nu \varphi = \partial_\nu \psi = 0$ on $\partial \Omega$.
\end{proof}

\section{Known properties of the solution}\label{sec:local_ex}
Next, we list some known properties of solutions to \eqref{CRu}--\eqref{IC} constructed in \cite{CLMR}.

\begin{lemma}\label{local_ex}
  Let
  \begin{align}\label{eq:local_ex:domain}
    \Omega \subset \R^n, n \in \N, \text{ be a smooth, bounded domain}
  \end{align}
  and let
  \begin{align}\label{eq:local_ex:init}
    u_0, v_0 \in W^{1, p}(\Omega) \text{ for some $p > n$} \quad \text{with} \quad 0 \not\equiv u_0 \ge 0 \text{ and } v_0 \ge 0 \text{ in $\Omega$}.
  \end{align}
  Then the system \eqref{CRu}--\eqref{IC} has a maximal unique classical solution
  \begin{align}\label{eq:local_ex:reg}
    (u, v) \in C^0([0, \tmax); \sobo1p) \cap C^\infty(\bar \Omega \times (0, \tmax))
  \end{align}
  and if $\tmax < \infty$, then $\limsup_{t \nearrow \tmax} (\|u(\cdot, t)\|_{L^\infty(\Omega)} + \|v(\cdot, t)\|_{L^\infty(\Omega)}) = \infty$.
  Moreover, $u$ and $v$ are positive in $\Ombar \times (0, \tmax)$.
\end{lemma}
\begin{proof}
  The existence of a $C^\infty(\bar \Omega \times (0, \tmax))$ solution that is nonnegative has been proved in \cite[Theorem~2.1]{CLMR}. Applying the Hopf lemma to each of the equations in \eqref{CRu} separately (in view of the regularity \eqref{eq:local_ex:reg}, $\Delta v$ is bounded on $\bar \Omega \times [\tau, T]$ for every $0 < \tau < T < \tmax$), due to zero Neumann boundary data \eqref{BC}, we arrive at the positivity of $u$ and $v$ for each $t \in (0, \tmax)$.
%They have been show to be
%  and positivity follows from the strict maximum principle.
\end{proof}

As noted in \cite[(3)]{CLMR}, integrating both equations in \eqref{CRu} immediately ensures, that both solution components are uniformly in time bounded in $\leb1$.
\begin{lemma}\label{l1_bdd}
  Suppose the assumptions of Lemma~\ref{local_ex} hold. Then the solution $(u, v)$ of \eqref{CRu}--\eqref{IC} given by Lemma~\ref{local_ex} fulfills
  \begin{align*}
    \|u(\cdot, t)\|_{\leb1} &= \|u_0\|_{\leb1} \quad \text{and} \\
    \norm{v(\cdot, t)}_{L^1(\Omega)} &= \e^{-t}\lr{\norm{v_0}_{L^1(\Omega)} - \norm{u_0}_{L^1(\Omega)}} + \norm{u_0}_{L^1(\Omega)}
  \end{align*}
  for all $t \in (0, \tmax)$.
\end{lemma}

Moreover, \cite{CLMR} has identified a Lyapunov functional, which served as the main ingredient for solving the question of global existence the two-dimensional case.
\begin{lemma}\label{lyapunov}
 % Suppose the assumptions of Lemma~\ref{local_ex} hold. Then the solution $(u, v)$ of \eqref{CRu}--\eqref{IC} given by Lemma~\ref{local_ex} fulfills
 Under the assumptions of Lemma~\ref{local_ex} the solution  $(u,v)$ satisfies
  \begin{align}\label{LF}
  \frac{\d}{\d t}\lr{\int_\Omega{ u\log{u} } + \frac{1}{2}\int_\Omega{\abs{\nabla v}^2 }} = -\lr{\int_\Omega{\abs{\Delta v}^2 + \intom \abs{\nabla v}^2 + \intom \frac{\abs{\nabla u}^2}{u} }}
  \end{align}
  for all $t \in (0, \tmax)$. In particular,
  \begin{align}\label{I_intable}
    \int_{0}^{\tmax} \lr{\int_\Omega{\abs{\Delta v}^2 + \intom \abs{\nabla v}^2 + \intom \frac{\abs{\nabla u}^2}{u} }} < \infty.
  \end{align}
\end{lemma}
\begin{proof}
  The differential inequality \eqref{LF} is entailed in \cite[Lemma~2.2]{CLMR}, upon which \eqref{I_intable} results by an integration in time as the Lyapunov functional is bounded from below.
\end{proof}%

\section{The main estimate}\label{section_main}
This section contains our main contribution, a calculation of the evolution of the Fisher information along the trajectories of \eqref{CRu}--\eqref{IC}. It is related to the Bakry-\'Emery calculation, see \cite{BE},  applied however to a system of equations.

Throughout this section, we fix a domain and initial data fulfilling \eqref{eq:local_ex:domain} and \eqref{eq:local_ex:init} as well as the solution $(u, v)$ of \eqref{CRu}--\eqref{IC}, with maximal existence time $\tmax$  given by Lemma~\ref{local_ex}. Moreover, we denote
\begin{align}
\frac{\d}{\d t}\lr{\int_\Omega{{\abs{\Delta v}^2 + \intom \abs{\nabla v}^2 + \intom \frac{\abs{\nabla u}^2}{u}} }} \label{LF2}
&\sfed \frac{\d}{\d t}I(t).
\end{align}
Our aim is to obtain an estimate of $I$. Notice that $I$ is an extended version of the Fisher information. Indeed, in the case of a single heat equation, the quantity $\int_\Omega \frac{ |\nabla u|^2}{u}$ is called Fisher's information.
The following remark explains our strategy.
\begin{remark}\label{rem}
We note that the inequality $\dot{I} \le cI^2$ would imply boundedness of $I$ due to \eqref{I_intable}. Indeed, using Ladyzhenskaya's trick (see \cite{Lady}) we would have
\[ \frac{\d}{\d t}\lr{I(t)\e^{-\int_{0}^{t}{cI(s)\, \d s}}} \le 0. \]
\end{remark}

Below we formulate and prove our main contribution. It extends the calculation controlling the evolution of the Fisher information to the case of a system of equations.

\begin{lemma}\label{mainestimate}
For all $t \in (0, \tmax)$, the estimate
\begin{align*} \nonumber
\dot{I}(t) &\le -2\int_{\Omega}{u\abs{D^2\log{u}}^2} + 8\int_\Omega{\lr{\nabla \sqrt u}^TD^2v\lr{\nabla \sqrt u}} \\ %\label{Idotineq}
&\mathrel{\hphantom{\le}} -2 \intom |\nabla \Delta v|^2 - 4 \intom |\Delta v|^2 - 2 \intom |\nabla v|^2 + 2\int_{\Omega}{\nabla u\cdot\nabla v}
\end{align*}
holds.
\end{lemma}
\begin{proof}
We notice
\begin{align*}
\frac{\d}{\d t}\int_\Omega{\abs{\Delta v}^2} = 2\int_\Omega{\Delta{v}_t\Delta v} \quad \text{in}\quad (0, \tmax).
\end{align*}
From the second equation in $(\ref{CRu})$ we can substitute $\Delta v = v_t + v - u$ (equivalently, we can take the inner product of the second equation in $(\ref{CRu})$ with $\Delta v_t$), to get
\begin{align} \nonumber
\frac{\d}{\d t}\int_\Omega{\abs{\Delta v}^2} &= 2{\int_\Omega{\Delta v_t\lr{v_t + v - u}}} = -2{\int_\Omega{\nabla v_t \cdot \lr{\nabla v_t + \nabla v - \nabla u}}} \\ \label{lapderiv}
&= -2{\int_\Omega{\abs{\nabla v_t}^2}} - \frac{\d}{\d t}{\int_\Omega{\abs{\nabla v}^2}} + 2{\int_\Omega{\nabla v_t \cdot \nabla u}} \quad \text{in}\quad (0, \tmax).
\end{align}
From $(\ref{lapderiv})$ we get
\begin{align} \label{partderiv}
\frac{\d}{\d t}\lr{\int_\Omega{{\abs{\Delta v}^2 + \abs{\nabla v}^2}}} = -2{\int_\Omega{\abs{\nabla v_t}^2}} + 2{\int_\Omega{\nabla v_t \cdot \nabla u}} \quad \text{in}\quad  (0, \tmax)
\end{align}
and from $(\ref{LF2})$ and $(\ref{partderiv})$ we obtain %(notice that here the sign is accounted for, so we will get a bound from below for $I$, i.e., $I \ge \dots$)
\begin{align} \label{mainineq}
\dot{I}(t) = -2{\int_\Omega{\abs{\nabla v_t}^2}} + 2{\int_\Omega{\nabla v_t \cdot \nabla u}} + \frac{\d}{\d t}\int_\Omega{\frac{\abs{\nabla u}^2}{u}} \quad \text{for all}\quad t \in (0, \tmax).
\end{align}
Next, we compute the last term on the right-hand-side,
\begin{align} \label{FI}
\frac{\d}{\d t}\int_\Omega{\frac{\abs{\nabla u}^2}{u}} = 4\frac{\d}{\d t}\int_\Omega{\abs{\nabla{\sqrt{u}}}^2} = 8\int_\Omega{\nabla{\vro_t}\cdot\nabla{\vro}} \quad \text{in}\quad (0, \tmax),
\end{align}
where we applied the substitution $\vro \defs \sqrt{u}$.
From the first equation in $(\ref{CRu})$ we have
\begin{align} \label{rhot2}
\varrho_t &= \frac{u_t}{2\sqrt{u}} = \frac{\Delta u + \nabla u\cdot\nabla v + u\Delta v}{2\sqrt{u}} = \frac{\Delta u}{2\sqrt{u}} + \nabla\vro\cdot\nabla v + \frac{1}{2}\vro\Delta v,
\end{align}
where
\begin{align} \label{Lapvro}
\Delta\vro &= \div\lr{\nabla\vro} = \div\lr{\frac{\nabla u}{2\sqrt{u}}} = \frac{\Delta u}{2\sqrt{u}} - \frac{\abs{\nabla u}^2}{4 u^{3/2}} = \frac{\Delta u}{2\sqrt{u}} - \frac{\abs{\nabla\vro}^2}{\vro}
\end{align}
in $\Omega \times (0, \tmax)$.
So, by plugging $(\ref{Lapvro})$ into $(\ref{rhot2})$, we find
\begin{align} \nonumber
\varrho_t &= \Delta\vro + \frac{\abs{\nabla\vro}^2}{\vro} + \nabla\vro\cdot\nabla v + \frac{1}{2}\vro\Delta v \quad \text{in} \quad \Omega \times (0, \tmax).
\end{align}
Hence, $(\ref{FI})$ becomes
\begin{align} \nonumber
\frac{\d}{\d t}\int_\Omega{\frac{\abs{\nabla u}^2}{u}} &= 8\int_\Omega{\nabla{\vro}\cdot\nabla\lr{\Delta\vro + \frac{\abs{\nabla\vro}^2}{\vro} + \nabla\vro\cdot\nabla v + \frac{1}{2}\vro\Delta v}} \\ \nonumber
&=8\lra{\int_\Omega{\nabla\vro\cdot\nabla\lr{\Delta\vro}} + \int_\Omega{2\frac{\lr{\nabla\vro}^TD^2\vro\lr{\nabla\vro}}{\vro}} - \int_\Omega{\frac{\abs{\nabla\vro}^4}{\vro^2}}} \\ \nonumber
&\mathrel{\hphantom{=}}+ 8\lra{\int_\Omega{\lr{\nabla\vro}^TD^2\vro\lr{\nabla v}} + \int_\Omega{\lr{\nabla\vro}^TD^2v\lr{\nabla\vro}}} \\ \label{FI2}
&\mathrel{\hphantom{=}}+ 4\int_{\Omega}{\nabla\lr{\vro\Delta v}\cdot\nabla\vro} \quad \text{in} \quad (0, \tmax).
\end{align}
Due to the Bochner formula $(\ref{bochnerformula})$,
\[ \nabla\vro\cdot\nabla\lr{\Delta\vro} = -\abs{D^2\vro}^2 + \frac{1}{2}\Delta\lr{\abs{\nabla\vro}^2} \quad \text{in} \quad \Omega \times (0, \tmax),
\]
we get
\begin{align*}
&\mathrel{\hphantom{=}} \int_\Omega{\nabla\vro\cdot\nabla\lr{\Delta\vro}} + \int_\Omega{2\frac{\lr{\nabla\vro}^TD^2\vro\lr{\nabla\vro}}{\vro}} - \int_\Omega{\frac{\abs{\nabla\vro}^4}{\vro^2}} \\
&= -\int_{\Omega}{\abs{D^2\vro - \frac{\nabla\vro\otimes\nabla\vro}{\vro}}^2}
+ \frac{1}{2}\int_{\Omega}{\Delta\lr{\abs{\nabla\vro}^2}}
\quad \text{in} \quad (0, \tmax).
\end{align*}
We note that the boundary condition $\left.\frac{\partial u}{\partial \nu}\right\rvert_{\partial\Omega} = 0$ implies that
\[
\left.\frac{\partial\vro}{\partial \nu}\right\rvert_{\partial\Omega} = \left.\frac{\frac{\partial u}{\partial \nu}}{2\sqrt{u}}\right\rvert_{\partial\Omega} = 0,
\]
so that an integration by parts and an application of Lemma~\ref{gradBCEvans}, which is possible thanks to the convexity of the domain $\Omega$, yield
\[
\int_{\Omega}{\Delta\lr{\abs{\nabla\vro}^2}} = \int_{\partial\Omega}{\frac{\partial\lr{\abs{\nabla\vro}^2}}{\partial \nu}}
\le 0
\quad \text{in} \quad (0, \tmax).
\]
Plugging the above into $(\ref{FI2})$, we obtain
\begin{align} \nonumber
\frac{\d}{\d t}\int_\Omega{\frac{\abs{\nabla u}^2}{u}} &\le -2\int_{\Omega}{u\abs{D^2\log{u}}^2} \\ \nonumber
&+ 8\lra{\int_\Omega{\lr{\nabla\vro}^TD^2\vro\lr{\nabla v}} + \int_\Omega{\lr{\nabla\vro}^TD^2v\lr{\nabla\vro}}} \\ \label{FI3}
&+ 4\int_{\Omega}{\nabla\lr{\vro\Delta v}\cdot\nabla\vro}
\quad \text{in} \quad (0, \tmax),
\end{align}
where we also used the relations
\[ \int_{\Omega}{\abs{D^2\vro - \frac{\nabla\vro\otimes\nabla\vro}{\vro}}^2} = \int_{\Omega}{\vro^2\abs{D^2\log{\vro}}^2} = \frac14 \intom u \abs{D^2 \log u}^2\]
in the first term on the right-hand side.

We now focus on the last term in $(\ref{FI3})$,
\begin{align} \label{gradLapv}
4\int_{\Omega}{\nabla\lr{\vro\Delta v}\cdot\nabla\vro} = 4\int_{\Omega}{\abs{\nabla\vro}^2\Delta v} + 4\int_{\Omega}{\vro\nabla\vro\cdot\nabla\lr{\Delta v}}.
\end{align}
Integration by parts yields
\begin{equation}  \label{D2vro}
4\int_{\Omega}{\abs{\nabla\vro}^2\Delta v} = -4\int_{\Omega}{\nabla\lr{\abs{\nabla\vro}^2}\cdot\nabla v}= -8\int_\Omega{\lr{\nabla\vro}^TD^2\vro\lr{\nabla v}}
\end{equation}
in $(0, \tmax)$.
For the second term in $(\ref{gradLapv})$ we substitute $\Delta v = v_t + v - u$ from the second equation in $(\ref{CRu})$ to obtain
\begin{align} \nonumber
&\mathrel{\hphantom{=}}4\int_{\Omega}{\vro\nabla\vro\cdot\nabla\lr{\Delta v}} = 2\int_{\Omega}{\nabla\vro^2\cdot\nabla\lr{v_ t + v - u}} \\ \label{gradLapv2}
&= 2\int_{\Omega}{\nabla u \cdot\nabla v_t} + 2\int_{\Omega}{\nabla u\cdot\nabla v} - 2\int_{\Omega}{\abs{\nabla u}^2}
\quad \text{in} \quad (0, \tmax).
\end{align}
Inserting $(\ref{D2vro})$ and $(\ref{gradLapv2})$ in $(\ref{FI3})$ gives
\begin{align} \nonumber
\frac{\d}{\d t}\int_\Omega{\frac{\abs{\nabla u}^2}{u}} &\le -2\int_{\Omega}{u\abs{D^2\log{u}}^2} + 8\int_\Omega{\lr{\nabla\vro}^TD^2v\lr{\nabla\vro}} \\ \nonumber
&\mathrel{\hphantom{=}}+ 2\int_{\Omega}{\nabla u \cdot\nabla v_t} + 2\int_{\Omega}{\nabla u\cdot\nabla v} - 2\int_{\Omega}{\abs{\nabla u}^2}
\quad \text{in} \quad (0, \tmax).
\end{align}
Therefore, going back to $(\ref{mainineq})$, we have
\begin{align} \nonumber
\dot{I}(t) &\le -2\int_{\Omega}{u\abs{D^2\log{u}}^2} + 8\int_\Omega{\lr{\nabla \sqrt u}^TD^2v\lr{\nabla \sqrt u}} \\ \nonumber
&\mathrel{\hphantom{=}}- 2{\int_\Omega{\abs{\nabla v_t}^2}} + 4{\int_\Omega{\nabla v_t \cdot \nabla u}} + 2\int_{\Omega}{\nabla u\cdot\nabla v} - 2\int_{\Omega}{\abs{\nabla u}^2}
\end{align}
for all $t \in (0, \tmax)$.
Since
\begin{align*}
  &\mathrel{\hphantom{=}}  - 2{\int_\Omega{\abs{\nabla v_t}^2}} + 4{\int_\Omega{\nabla v_t \cdot \nabla u}} -  2\int_{\Omega}{\abs{\nabla u}^2} \\
  &=  -2 \intom |\nabla (v_t - u)|^2
   =  -2 \intom |\nabla (\Delta v - v)|^2 \\
  &=  -2 \intom |\nabla \Delta v|^2
      -4 \intom |\Delta v|^2
      -2 \intom |\nabla v|^2
  \quad \text{in} \quad (0, \tmax),
\end{align*}
we obtain the desired estimate.
\end{proof}

Next, we simplify the previous differential inequality, which will allow us to argue in a more straightforward manner in the sequel.
\begin{lemma}\label{mainest2}
  Throughout $(0, \tmax)$, it holds that
  \begin{align*}
   &\mathrel{\hphantom{=}} \frac{\d}{\d t} \left( 4 \int_\Omega |\nabla \sqrt u|^2 + \int_\Omega |\Delta v|^2 \right) \\
    &\le  - 2 \int_\Omega u |D^2 \log u|^2
          - 2 \int_\Omega |\nabla \Delta v|^2
          - 2 \int_\Omega |\Delta v|^2 \\
    &\mathrel{\hphantom{=}}     + 8 \int_\Omega (\nabla \sqrt u)^T D^2 v \nabla \sqrt u.
  \end{align*}
\end{lemma}
\begin{proof}
  This follows immediately from Lemma~\ref{mainestimate} and the fact that
  \begin{align*}
        \frac{\d}{\d t} \int_\Omega |\nabla v|^2
    &=  - 2 \int_\Omega \Delta v v_t
     =  - 2 \int_\Omega |\Delta v|^2
        - 2 \int_\Omega |\nabla v|^2
        + 2 \int_\Omega \nabla v \cdot \nabla v
  \end{align*}
  in $(0, \tmax)$.
\end{proof}

\section{Proof of the main theorem}\label{fourth}
We are now in a position to utilize our calculation from Lemma~\ref{mainestimate} and complete the proof of the announced result.
As in the previous section, we fix a domain $\Omega$ and initial data $u_0, v_0$ satisfying \eqref{eq:local_ex:domain} and \eqref{eq:local_ex:init} as well as the solution $(u, v)$ of \eqref{CRu}--\eqref{IC} given by Lemma~\ref{local_ex}.
Moreover, as the solution is unique by Lemma~\ref{local_ex}, $\tmax$ is infinite if and only if the solution with initial data $(u(\cdot, t_0), v(\cdot, t_0))$ for some $t_0 \in (0, \tmax)$ exists globally.
Thus, by switching to the solution with these initial data and recalling \eqref{eq:local_ex:reg}, we may assume $u, v \in C^\infty(\Ombar \times [0, \tmax))$.

The following lemma is the first step in a bootstrapping procedure yielding the required regularity of the solution.
\begin{lemma}\label{u_l3}
  Suppose $n=3$ and that $\Omega$ is convex. Let $T \in (0, \tmax] \cap (0, \infty)$ and suppose that $\nabla \sqrt{u} \in L^4(0, T; L^2(\Omega))$.
  Then there is $C > 0$ such that
  \begin{align*}%\label{eq:u_l3}
    \intom u^3(\cdot, t) \le C \qquad \text{for all}\quad  t \in (0, T).
  \end{align*}
\end{lemma}
\begin{proof}
We define
$$ J(t) \defs 4 \int_\Omega |\nabla \sqrt u|^2 + \int_\Omega |\Delta v|^2 . $$
Taking $\varphi = u$, $\psi = v$ and $\varepsilon = \frac14$ in Lemma~\ref{lemmaone} and making use of Lemma~\ref{mainest2}, we arrive at
\begin{align*}%\label{finalf3}
  \dot J(t)
&\le c_1\lr{\int_\Omega{\abs{\nabla \sqrt u}^2}}^{3} + c_1
 \le c_1\lr{\int_\Omega{\abs{\nabla \sqrt u}^2}}^{2} J(t) + c_1
\end{align*}
in $(0, T)$ for some $c_1 > 0$.
Thus, with $K(t) \defs c_1 \int_0^t \lr{ \int_\Omega{\abs{\nabla \sqrt u}^2}}^{2}$, $t \in (0, T)$, we have
\begin{align*}
      J(t)
  \le \e^{K(t)} J(0) + c_1 \int_0^t \e^{K(t-s)} \,\d s
  \le \e^{K(T)} J(0) + c_1 T \e^{K(T)}
\end{align*}
for all $t \in (0, T)$. Since $K(T) < \infty$ by assumption, we obtain boundedness of $\sup_{t \in (0, T)} \|\nabla \sqrt{u(\cdot, t)}\|_{L^2(\Omega)}$,
which, in conjunction with Lemma~\ref{l1_bdd}, implies the desired estimate as $W^{1, 2}(\Omega)$ embeds continuously into $L^6(\Omega)$.
\end{proof}

Next, we show the higher regularity of the obtained solution.

\begin{lemma}\label{u_linfty}
%  Suppose $n=3$ and that $\Omega$ is convex. Let $T \in (0, \tmax] \cap (0, \infty)$ and suppose that $\nabla \sqrt{u} \in L^4(0, T; L^2(\Omega))$.
Under the assumptions of Lemma~\ref{u_l3} there is $C > 0$ such that
  \begin{align*}
    \|u(\cdot, t)\|_{L^\infty(\Omega)} + \|v(\cdot, t)\|_{L^\infty(\Omega)} \le C \qquad \text{for all}\quad t \in (0, T).
  \end{align*}
\end{lemma}
\begin{proof}
  We fix $3 < r < q < \infty$.
  Making use of well-known semigroup estimates (cf.\ \cite[Lemma~1.3~(ii) and (iii)]{lp_lq}), we obtain
  \begin{align*}
    &\mathrel{\hphantom{\le}} \|\nabla v(\cdot, t)\|_{\leb q} \\
    &     \le  \|\nabla \ure^{t (\Delta-1)} v_0\|_{\leb q}
          + \int_0^t \|\ure^{(t-s) (\Delta-1)} u(\cdot, s) \|_{\leb q} \ds \\
    &\le  c_1 \ure^{-t} \|\nabla v_0\|_{\leb q}
          + c_2 \int_0^t \left(1+(t-s)^{-\frac12 - \frac32 (\frac13 - \frac1q)}\right) \ure^{-(t-s)} \|u(\cdot, s)\|_{\leb 3} \ds \\
    &\le  c_1 \|\nabla v_0\|_{\leb q}
           + c_2 \sup_{s \in (0, T)} \|u(\cdot, s)\|_{\leb 3} \int_0^{T} \left(1+s^{-\frac12 - \frac32 (\frac13 - \frac1q)}\right) \ds
  \end{align*}
  for all $t \in (0, T)$ and some $c_1, c_2 > 0$.
  Since $-\frac12 - \frac32 (\frac13 - \frac1q) > -1$ and recalling Lemma~\ref{u_l3},
  we conclude that there is $c_3 > 0$ such that $\|\nabla v(\cdot, t)\|_{\leb q} \le c_3$ for all $t \in (0, T)$.
  Since $q > 3$, $W^{1, q}(\Omega)$ embeds continuously into $\leb\infty$ and so the above estimate in conjunction with Lemma~\ref{l1_bdd} imply that $\sup_{t \in (0, T)} \|v(\cdot, t)\|_{\leb \infty}$ is finite as well.

  Relying on the maximum principle and again on well-known semigroup estimates (cf.\ \cite[Lemma~1.3~(iv)]{lp_lq}), we further estimate
  \begin{align*}
    &\mathrel{\hphantom{\le}}\left\|u(\cdot, t)\right\|_{\leb\infty} \\
    &\le  \|\ure^{t \Delta} u_0\|_{\leb \infty}
          + \int_0^t \left\|\ure^{(t-s) \Delta} \nabla \cdot (u \nabla v)(\cdot, s) \ds \right\|_{\leb\infty} \\
    &\le  \left\| u_0\right\|_{\leb \infty}
          + c_4 \int_0^t \left(1 + (t-s)^{-\frac12 - \frac32 (\frac1r-\frac1\infty)}\right) \|(u \nabla v)(\cdot, s)\|_{\leb r} \ds \\
    &\le  \left\| u_0\right\|_{\leb \infty}
          + c_4 \sup_{s \in (0, t)} \|(u \nabla v)(\cdot, s)\|_{\leb r} \int_0^T \left(1 + s^{-\frac12 - \frac3{2r}}\right) \ds
  \end{align*}
  for all $t \in (0, T)$ and some $c_4 > 0$.
  Since with $\lambda \defs \frac{rq}{q-r}$ and $\theta \defs \frac{\lambda-1}{\lambda} \in (0, 1)$ we have
  \begin{align*}
          \|(u \nabla v)(\cdot, s)\|_{\leb r}
    &\le  \|u(\cdot, s)\|_{\leb{\lambda}} \|\nabla v(\cdot, s)\|_{\leb q} \\
    &\le  \|u(\cdot, s)\|_{\leb\infty}^\theta \|u(\cdot, s)\|_{\leb1}^{1-\theta} \|\nabla v(\cdot, s)\|_{\leb q}
  \end{align*}
  for all $s \in (0, T)$, we see that there is $c_5 > 0$ such that
  \begin{align*}
          \|u(\cdot, t)\|_{\leb\infty}
    &\le  c_5 + c_5 \sup_{s \in (0, t)} \|u(\cdot, s)\|_{\leb\infty}^\theta
  \end{align*}
  for all $t \in (0, T)$.
  For $t \in (0, T)$, we set $A(t) \defs \sup_{s \in (0, t)} \|u(\cdot, t)\|_{\leb\infty}$.
  Since $\theta \in (0, 1)$, by means of Young's inequality we obtain $A(t) \le c_5 + c_5 A^\theta(t) \le c_6 + \frac12 A(t)$ for some $c_6 > 0$ (not depending on $t$)
  and hence also $A(t) \le 2c_6$ for all $t \in (0, T)$.
  Taking $t \nearrow T$ shows that also $u$ remains bounded in $\Omega \times (0, T)$.
\end{proof}

We are now in position to prove Theorem $\ref{mainresult}$.

\begin{proof}[Proof of Theorem~\ref{mainresult}]
  Suppose that $\tmax < \infty$, then Lemma~\ref{u_linfty} asserts boundedness of $u$ and $v$ in $\Omega \times (0, \tmax)$.
  However, this contradicts the extensibility criterion in Lemma~\ref{local_ex}.
\end{proof}

\section{Conclusion} \label{s:conclusion}
On the one hand, we obtained a condition which guarantees global existence of solutions for the chemorepulsion system in three-dimensional space. This result can be improved, as suggested to us by M.~Winkler and an anonymous referee, see Appendix~\ref{sec:alt_proof}. On the other hand, we notice from our computations that the concavity of the function $v$ would greatly simplify our argument. It would lead to boundedness of the function $I(t)$, and hence to the global existence of solutions, as shown in this paper. Indeed, we see from $(\ref{notconvex})$ that if the function $v$ is concave, i.e., its Hessian is negative-semidefinite,
\[ x^T\, D^2v \, x \le 0 \qquad \mbox{for every} \qquad x \in \R^n,
\]
we have the following differential inequality for $I(t)$
\begin{align} \nonumber
\dot{I}(t) &\le 0.
\end{align}

%\begin{remark} \label{concave}
%
%We see from $(\ref{notconvex})$ that if the function $v$ is concave, i.e.\
%
%\[ x^T\, D^2v \, x \le 0 \qquad \mbox{for every} \qquad x \in \R^n,
%\]
%we have the following differential inequality
%\begin{align} \nonumber
%\dot{I}(t) &\le 2\int_{\Omega}{\nabla u\cdot\nabla v}.
%\end{align}
%\end{remark}

From the equation \eqref{CRu} we see that $v$ is not far from being concave. Taking a simplified version of  the equation for $v$ [assuming $v_t = 0$ and neglecting $v$ on the right-hand side of the second equation in \eqref{CRu}], we have
\[
\Delta v = - u \le 0, %\qquad \mbox{iff} \qquad v \le u.
\]
which would also hold if the Hessian of $v$ was negative-semidefinite.

Verifying concavity of a solution of a parabolic boundary value problem posed on a convex domain has been studied before. In the context of one parabolic equation of certain type, some positive results can be found in \cite{Korevar}, for example. However, we are not aware of any result in this direction for systems of equations.

\appendix
\section{A new inequality} \label{s:bernis}

As a byproduct of our arguments, we discovered a differential inequality relating the second norm of the Hessian of the square-root of a positive function with the dissipation of the Fisher information along the heat flow. Due to the fact that both of these quantities appear in the calculation of the evolution of the Fisher information along the heat flow, the following inequality is interesting in its own right and may have further applications. In particular, it was used by the first author in \cite{BC}, where global-in-time regular unique solution to a $1$D thermoelasticity system is obtained. Next, a very recent application of the inequality \eqref{cieslakineq} led to the interesting result in the theory of $1$D combustion, see \cite{LY}.

\begin{lemma}\label{cieslakinequality}
Let $\Omega \subset \R^n$ be a smooth bounded domain. For every positive function $u \in C^2(\Ombar)$ with the boundary condition $\left.\frac{\partial u}{\partial \nu}\right\rvert_{\partial \Omega} = 0$, we have
\begin{equation}\label{cieslakineq}
\int_{\Omega}{\abs{D^2\sqrt{u}}^2} \le C\int_{\Omega}{u\abs{D^2\log{u}}^2},
\end{equation}
where $C = 1 + \frac{\sqrt{n}}{2} + \frac{n}{8}$.
\end{lemma}

\begin{remark}
We note that the inequality $(\ref{cieslakineq})$ does not hold pointwise, i.e., there is no constant $C > 0$ such that, for every positive $u \in C^2(\Ombar)$,
\[ \abs{D^2\sqrt{u}}^2 \le Cu\abs{D^2\log{u}}^2 \quad \text{in} \quad \Ombar. \]
\end{remark}

\begin{proof}[Proof of Lemma $\ref{cieslakinequality}$]
We first note that
\[ \lra{D^2\sqrt{u}}^2_{ij} = \lr{\frac{\partial_{x_ix_j}u}{2u^{1/2}} - \frac{\partial_{x_i}u\partial_{x_j}u}{4u^{3/2}}}^2 = \frac{1}{4}\lr{\frac{\partial_{x_ix_j}u}{u^{1/2}} - \frac{1}{2}\frac{\partial_{x_i}u\partial_{x_j}u}{u^{3/2}}}^2 \]
and
\[ u\lra{D^2\log{u}}^2_{ij} = u\lr{\frac{\partial_{x_ix_j}u}{u} - \frac{\partial_{x_i}u\partial_{x_j}u}{u^2}}^2 = \lr{\frac{\partial_{x_ix_j}u}{u^{1/2}} - \frac{\partial_{x_i}u\partial_{x_j}u}{u^{3/2}}}^2
\]
in $\Ombar$.
Using the simple fact that $\lr{a + b}^2 \le 2\lr{a^2 + b^2}$, we get
\begin{align}\nonumber
\lra{D^2\sqrt{u}}^2_{ij} &= \frac{1}{4}\lr{\frac{\partial_{x_ix_j}u}{u^{1/2}} - \frac{\partial_{x_i}u\partial_{x_j}u}{u^{3/2}} + \frac{1}{2}\frac{\partial_{x_i}u\partial_{x_j}u}{u^{3/2}}}^2 \\ \nonumber
&\le \frac{1}{4}\lra{2\lr{\frac{\partial_{x_ix_j}u}{u^{1/2}} - \frac{\partial_{x_i}u\partial_{x_j}u}{u^{3/2}}}^2 + \frac{1}{2}\lr{\frac{\partial_{x_i}u\partial_{x_j}u}{u^{3/2}}}^2} \\ \nonumber
&= \frac{1}{2}{u\lra{D^2\log{u}}^2_{ij} + \frac{1}{8}\lr{\frac{\partial_{x_i}u\partial_{x_j}u}{u^{3/2}}}^2}.
\end{align}
Therefore, we obtain
\begin{align}\nonumber
{\abs{D^2\sqrt{u}}^2} &= {\sum_{i, j = 1}^n{\lra{D^2\sqrt{u}}^2_{ij}}} \le {\sum_{i, j = 1}^n{\lra{\frac{1}{2}{u\lra{D^2\log{u}}^2_{ij} + \frac{1}{8}\lr{\frac{\partial_{x_i}u\partial_{x_j}u}{u^{3/2}}}^2}}}}  \\ \label{cieslakfinal}
&= \frac{1}{2}{u\abs{D^2\log{u}}^2} + \frac{1}{8}{\frac{\abs{\nabla u}^4}{u^3}}
\quad \text{in}\quad \Ombar.
\end{align}
Applying Lemma $\ref{winkler}$ to $(\ref{cieslakfinal})$, we get
\begin{align}\nonumber
\int_{\Omega}{\abs{D^2\sqrt{u}}^2} \le \lr{\frac{1}{2} + \frac{1}{8}\lr{2 + \sqrt{n}}^2}\int_{\Omega}{u\abs{D^2\log{u}}^2},
\end{align}
as required.
\end{proof}

\section{Alternative proof}\label{sec:alt_proof}
This section is devoted to the presentation of the result communicated to us by M.\,Winkler as well as one of the anonymous referees. It gives
an alternative conditional result, extending our Theorem~\ref{mainresult}. On the one hand, no convexity of the domain is required, on the other hand, only a zero order estimate of $u$ is required. 

It is based on the well-known fact that in dimension $3$, bounding the \linebreak $L^\infty((0,\tmax); L^p(\Omega))$ norm of $u$ for any $p > \frac32$ allows prolongation of a solution to \eqref{CRu}, see for instance \cite{BBTW}.  

\begin{lemma}
Let $\Omega \subset \R^3$ be a smooth, bounded domain and let $u_0, v_0$ be as in \eqref{eq:local_ex:init}. Suppose that the maximal existence time $\tmax$ of the solution $(u, v)$ to \eqref{CRu}--\eqref{IC} given by Lemma~\ref{local_ex} is finite. Then
\begin{align} \label{l2l3}
\int_0^{\tmax}{\norm{u(t)}_{L^3(\Omega)}^2 \, \mathrm{d}t} = \infty.
\end{align}
In particular, 
$$ \int_0^{\tmax}{\norm{\nabla\sqrt{u(t)}}_{L^2(\Omega)}^4 \, \mathrm{d}t} = \infty. $$
\end{lemma}

\begin{proof}
In view of the Gagliardo--Nirenberg inequality and conservation of mass, we only need to verify \eqref{l2l3}. To this end, we suppose that \eqref{l2l3} does not hold, which allows us to first apply maximal Sobolev regularity theory to the second equation in \eqref{CRu} to see that
\begin{align} \label{hest}
\int_{\frac\tmax2}^{\tmax}{\norm{\Delta v(t)}_{L^3(\Omega)}^2 \, \mathrm{d}t} < \infty.
\end{align}
Then, we use the first equation in \eqref{CRu} along with an integration by parts to get
\begin{align} \nonumber
\frac{1}{2}\frac{\mathrm{d}}{\mathrm{d}t}\int_{\Omega}{u^2} + \int_{\Omega}{\abs{\nabla u}^2} + \int_{\Omega}{u^2} &\le \frac{1}{2}\int_{\Omega}{u^2\Delta v} + \int_{\Omega}{{u}^2} \\ \nonumber
&\le \frac{1}{2}\norm{u}_{L^3(\Omega)}^2\norm{\Delta v}_{L^3(\Omega)} + \int_{\Omega}{{u}^2}
\end{align}
in $(0, \infty)$.
Next, applications of the Gagliardo--Nirenberg interpolation and Young’s inequality give for all $t \in (0, \tmax)$,
\begin{align} \nonumber
\frac{1}{2}\frac{\mathrm{d}}{\mathrm{d}t}\int_{\Omega}{u^2} + \int_{\Omega}{\abs{\nabla u}^2} + \int_{\Omega}{u^2} \le \int_{\Omega}{\lr{\abs{\nabla u}^2 + u^2}} + c_1\lr{\norm{\Delta v}_{L^3(\Omega)}^2 + 1}\norm{u}_{L^2(\Omega)}^2
%&\le c_1\norm{\Delta v}_{L^3(\Omega)}\lr{ \int_{\Omega}{\abs{\nabla u}^2 + u^2}}^{1/2}\lr{ \int_{\Omega}{u^2}}^{1/2} +  \int_{\Omega}{u^2}
\end{align}
for some constant $c_1 > 0$.% and $c_2 > 0$.

Writing $y(t) \defs \int_{\Omega}{u^2(\cdot, t)}$ and $h(t) \defs 2c_1\lr{\norm{\Delta v(\cdot, t)}_{L^3(\Omega)}^2 + 1}$ for $t \in (0, \tmax)$, we arrive at a differential inequality
\begin{equation} \label{diffh}
\dot{y}(t) \le h(t)y(t) \qquad \mbox{for all } t \in (0, \tmax).
\end{equation}
Since $h \in L^1(\frac\tmax2, \tmax)$ by \eqref{hest}, integration in time of \eqref{diffh} shows boundedness of $u$ in $L^\infty((0, \tmax); \leb2)$, contradicting finiteness of $\tmax$.
\end{proof}

\section*{Acknowledgments}
T.C.\ was supported by the National Science Center of Poland grant SONATA BIS 7 number UMO-2017/26/E/ST1/00989. K.H.\ was partially supported by the National Science Center of Poland grant SONATA BIS 10 number UMO-2020/38/E/ST1/00469. M.S.\,was supported by the National Science Center of Poland grant SONATA BIS 10 number UMO-2020/38/E/ST1/00596. The authors are grateful to both anonymous referees as well as to Michael Winkler for careful reading of the manuscript and suggestions leading to an essential improvement of our results. In particular, Appendix~\ref{sec:alt_proof} was observed by M.~Winkler and the anonymous referee.

%%%%%%%%%%%%%%%%%%%%%%%%%%%%%%%%%%%%%%%%%%%%%%%%%%%%%%%%%%%%%%%%%
%\bibliographystyle{amsplain}
\bibliographystyle{acm}
%\bibliography{biblio}

\begin{thebibliography}{99}

\bibitem{BE}  Bakry, D. and \'{E}mery, M. \textit{Diffusions hypercontractives}. S\'{e}minaire de probabilit\'{e}s, {XIX}, 1983/84, pp. 177--206, Lecture Notes in Math., 1123, Springer, Berlin, 1985.

\bibitem{BBTW} Bellomo, N., Bellouquid, A., Tao, Y. and Winkler, M.
\textit{Toward a mathematical theory of Keller--Segel models of pattern formation in biological tissues}.
Math. Models Methods Appl. Sci. 25, 09 (2015), 1663--1763.

\bibitem{BC} Bies, P.M., Cie\'slak, T. \textit{Global-in-time regular unique solutions with positive temperature to the $1$d thermoelasticity}, to appear in SIAM J. Math. Anal., arXiv:2303.12522.


\bibitem{CF} Cie\'slak, T., and Fujie, K. \textit{Some remarks on well-posedness of the higher-dimensional
chemorepulsion system}. Bull. Pol. Acad. Sci. Math. 67, 2 (2019), 165--178.

\bibitem{CLMR}
Cie\'slak, T., Lauren\c cot, Ph., and Morales-Rodrigo, C. \textit{Global existence and convergence
to steady states in a chemorepulsion system}. In Parabolic and Navier--Stokes equations. Part
1, vol. 81 of Banach Center Publ. Polish Acad. Sci. Inst. Math., Warsaw, 2008, pp. 105--117.

%\bibitem{DPGG}
%Dal Passo, R., Garcke, H., and Grun, G. \textit{On a fourth-order degenerate parabolic equation:
%global entropy estimates, existence, and qualitative behavior of solutions}. SIAM J. Math.
%Anal. 29, 2 (1998), 321--342.

\bibitem{E}
Evans, L. C. \textit{Entropy and Partial Differential Equations}. CreateSpace Independent Publishing Platform, 2014.

\bibitem{Freitag}
Freitag, M. \textit{Global existence and boundedness in a chemorepulsion system with superlinear
diffusion}. Discrete Contin. Dyn. Syst. 38, 11 (2018), 5943--5961.

\bibitem{Friedman}
Friedman, A. \textit{Partial differential equations}, original ed. Robert E. Krieger Publishing Co.,
Huntington, N.Y., 1976.

\bibitem{Fuest}
Fuest, M. \textit{Global solutions near homogeneous steady states in a multidimensional population
model with both predator- and prey-taxis}. SIAM J. Math. Anal. 52, 6 (2020), 5865--5891.

\bibitem{Korevar}
Korevaar, N.J. \textit{Convex solutions to nonlinear elliptic and parabolic boundary value problems}. Indiana Univ. Math. J. 32 (1983), no. 4, 603--614.

\bibitem{Lady}
Ladyzenskaja, O. A. \textit{Solution ``in the large'' of the nonstationary boundary value problem
for the Navier--Stokes system with two space variables}. Comm. Pure Appl. Math. 12 (1959),
427--433.

\bibitem{LW}
Lankeit, J., Winkler, M. \textit{Facing low regularity in chemotaxis systems}. Jahresber. Dtsch. Math.-Ver. 122 (2020), 35–64.

\bibitem{LY} Li, S., Yang, J. \textit{Characterisations for the depletion of reactant in a one-dimensional dynamic combustion model}, arXiv:2308.16506.

\bibitem{lin_mu}
Lin, K. Mu, C. \textit{Global existence and convergence to steady states for an attraction-repulsion chemotaxis system}. Nonlinear Anal. Real World Appl. 31 (2016), 630–642.

\bibitem{luca}
Luca, M., Chavez-Ross, A. Edelstein-Keshet, L., Mogilner, A. 
\textit{Chemotactic signaling, microglia, and Alzheimer’s disease senile plaques: Is there a connection?}  Bull. Math. Biol. 65, (2003), 693–730.

\bibitem{nagaial}
Nagai, T., Seki, Y., Yamada, T. \textit{Global existence of solutions to a parabolic attraction-repulsion chemotaxis system in $\R^2$: the attractive dominant case}. Nonlinear Anal. Real World Appl. 62 (2021), Paper No. 103357, 16 pp.

\bibitem{Tao}
Tao, Y. \textit{Global dynamics in a higher-dimensional repulsion chemotaxis model with nonlinear
sensitivity}. Discrete Contin. Dyn. Syst. Ser. B 18, 10 (2013), 2705--2722.

\bibitem{TaoWang}
Tao, Y., and Wang, Z.-A. \textit{Competing effects of attraction vs. repulsion in chemotaxis}. Math.
Models Methods Appl. Sci. 23, 1 (2013), 1--36.

\bibitem{lp_lq}
Winkler, M. \textit{Aggregation vs.\ global diffusive behavior in the higher-dimensional Keller--Segel model}.
J. Differ. Equ. 248, 12 (2010), 2889--2905,

\bibitem{Winkler}
Winkler, M. \textit{Global large-data solutions in a chemotaxis-(Navier--)Stokes system modeling
cellular swimming in fluid drops}. Comm. Partial Differential Equations 37, 2 (2012), 319--351.

\end{thebibliography}

\end{document}